\documentclass[letterpaper, 10 pt, conference]{ieeeconf}  % Comment this line out if you need a4paper

\IEEEoverridecommandlockouts                              % This command is only needed if 
\usepackage{cite}
\usepackage{float}

\usepackage{graphicx}
\usepackage{amssymb}
\usepackage[cmex 10]{amsmath}
\usepackage{array}
\usepackage{mdwmath}

\newtheorem{theorem}{Theorem}
\newtheorem{definition}{Definition}
\newtheorem{lemma}{Lemma}

\newtheorem{proposition}{Proposition}
\usepackage{subcaption}
\usepackage{epsfig}
\newcommand*{\QEDB}{\hfill\ensuremath{\square}}%

\usepackage{times}

\usepackage{tikz}
\usetikzlibrary{shapes,arrows}

\usetikzlibrary{arrows,calc}

\tikzset{
        block/.style = {draw, rectangle, 
                        minimum height=1cm, 
                        minimum width=2cm},
        input/.style = {coordinate,node distance=1cm},
        output/.style = {coordinate,node distance=4cm},
        arrow/.style={draw, -latex,node distance=2cm},
        pinstyle/.style = {pin edge={latex-, black,node distance=2cm}},
        sum/.style = {draw, circle, node distance=1cm}
}

\title{Passivity Analysis of Higher Order Evolutionary Dynamics and Population Games}

\author{M. A. Mabrok and Jeff S. Shamma
\thanks{M. A. Mabrok (m.a.mabrok@gmail.com) and Jeff S. Shamma (jeff.shamma@kaust.edu.sa) are with the King Abdullah University of Science and Technology (KAUST),
Computer, Electrical and Mathematical Sciences and Engineering Division (CEMSE), Thuwal 23955--6900, Saudi Arabia. Research supported by funding from KAUST.}}

\begin{document}

\maketitle


\begin{abstract}

In population games, a large population of players, modeled as a continuum, is divided into subpopulations, and the fitness or payoff of each subpopulation depends on the overall population composition. Evolutionary dynamics describe how the population composition changes in response to the fitness levels, resulting in a closed-loop feedback system. Recent work established a connection between passivity theory and certain classes of population games, namely so-called ``stable games''. In particular, it was shown that a combination of stable games and (an analogue of) passive evolutionary dynamics results in stable convergence to Nash equilibrium. This paper considers the converse question of necessary conditions for evolutionary dynamics to exhibit stable behaviors for all generalized stable games. Here, generalization  refers to ``higher order" games where the population payoffs may be a dynamic function of the population state. Using methods from robust control analysis, we show that if an evolutionary dynamic does not satisfy a passivity property, then it is possible to construct a generalized stable game that results in instability. The results are illustrated on selected evolutionary dynamics with particular attention to replicator dynamics, which are also shown to be lossless, a special class of passive systems.  

\end{abstract}

\section{Introduction}

Population games  \cite{Sandholm2010, Hofbauer1998} model interactions among a large number of players, or agents, in which each agent's payoff or fitness depends on its own strategy and the distribution of strategies of other agents. There has been extensive research in a variety of settings, ranging from societal \cite{Young1998} to biological \cite{Smith1982} to engineered \cite{Marden2015}.

A central question in population games, as well as the related topic of learning in games \cite{Fudenberg1998,Hart2005,Young2005},  is understanding the long run behavior of player strategies. In particular, under what conditions do population strategies converge to a solution concept such as Nash equilibrium? The outcome depends on both the underlying game and the particular evolutionary dynamics (e.g., \cite{Hofbauer2003}), and behaviors can range from convergence for classes of game/dynamics pairings \cite{Hofbauer2009} to chaos in seemingly simple settings \cite{Sato2002}. Furthermore, a specific game can exhibit inherent obstacles to convergence for broad classes of evolutionary dynamics \cite{Hart2003}. Contrary to Nash equilibrium, there are relaxed solution concepts, such as coarse correlated equilibria, that are universally (i.e., for all games) induced by various forms evolutionary dynamics \cite{Fudenberg1995, Hart2001}.

Of specific interest herein is the class of population games called \textit{stable games} \cite{Hofbauer2009}. These games exhibit a property called ``self-defeating externalities''. Whenever a segment of the population revises its strategies, the payoff gains in the adopted strategy are less than the payoff gains of the abandoned strategy. It was shown that the class of stable games results in convergence to Nash equilibrium when paired with a variety of evolutionary dynamics. Following work \cite{Fox2013} established a connection between stable games and passivity theory \cite{Willems1972}. Generally speaking, it was shown that stable games exhibit a property related to passivity. Furthermore,  various  evolutionary dynamics also exhibit a form of passivity. Accordingly, since interconnections of passive dynamical
systems exhibit stable behavior, one can conclude  that  passive evolutionary
dynamics coupled with stable games exhibit stable  behavior. 

The connection to passivity enables the opportunity to analyze in a similar way broader class of both games and evolutionary dynamics. Of particular interest here are higher order games and higher order dynamics. In the canonical models of population games, the fitness of various population strategies is a static function of the population composition. In a higher order model, this dependence can be dynamic, e.g., as a model of path dependencies \cite{Fox2013}. Likewise, in canonical forms of evolutionary dynamics, the number of states is equal to the number of population strategies. Higher order dynamics, through the introduction of auxiliary states, can exhibit qualitatively different behaviors. For example, 
instabilities \cite{Hart2003} or even chaos \cite{Sato2002} can be eliminated through modifications of standard evolutionary dynamics that reflect a form of myopic anticipation \cite{Shamma2005,Arslan2006}. Recent work has shown that higher order variants \cite{Laraki2013} of the well know replicator dynamics can lead to the elimination
of weakly dominated strategies, followed by the iterated deletion of strictly dominated strategies, a property not exhibited by standard replicator dynamics. 
  
This paper considers the following converse question: \textit{Under what conditions does a evolutionary dynamic stabilize all stable games?} In addressing this question, we will admit both higher order evolutionary dynamics and higher order stable games.  Using methods from robust control analysis, we show that if an evolutionary dynamic does not satisfy a passivity property, then it is possible to construct a higher order stable game that results in instability. The results are similar in spirit to prior work on the necessity of a small gain condition to stabilize certain classes of feedback interconnections \cite{Shamma1991,Shamma1993,Freeman2001}. 

The remainder of this paper is organized as follows: Section II presents preliminary material on population games and passivity. Section III establishes a necessity condition for stable interconnection with passive systems. Section IV specializes the results to population games and  presents illustrative simulations. Finally, Section V contains concluding remarks.

\section{Preliminaries and notations }\label{sec:PAN}

In this section,  some preliminaries and notations  form game  theory and passivity theory are provided  in order to establish our results. 
%Here the  focus will be  a class of games named \emph{passive games} introduced  in  \cite{Fox2013},  which is  a generalization for the  \textit{stable games} introduced in \cite{Hofbauer2009}.

\subsection{Passivity  theory}
Passivity implies  useful properties such as stability, and  the importance of passivity as tool in nonlinear control of interconnected systems---unlike Lyapunov stability criteria---relays  on the fact that any  set of  passive sub-systems in parallel or feedback configuration forms a passive  system.  In other words,  by ensuring that every subsystem is passive, a complex structure of subsystems  can be built  to satisfy certain properties.    

Consider $\Sigma$ to be a   nonlinear dynamical system with the following state space realization: 
\begin{align}\label{eq:nonlinearsys}
\dot{x} & =f(x,u)\\
y & =g(x,u), \notag
\end{align}
where, $u \in R^m$ is the system's input vector, $y \in R^m$ is the system's output vector and $x(t) \in R^n$ is the system's state vector. Next, we present  two definitions  for  passive system from both state space and input-output perspectives.

\begin{definition}
The nonlinear  system $\Sigma$ with state space \eqref{eq:nonlinearsys} is said to be passive if there exist  \textit{storage function } $L(x(t))$ such that:
\begin{align}\label{eq:storage1}
L(x(t))\leq L(x_0)+\int_0^t\!  u(\tau)^T y(\tau) d \tau.
\end{align}
\end{definition}

The input-output definition of passivity property is given as follows: 
\begin{definition}
The nonlinear  system $\Sigma$  is said to be passive if there exist  constant $\alpha$  such that:
\begin{align}\label{eq:inerpassive}
\langle {\Sigma}u, u\rangle_T \geq\alpha, \forall u \in U, T\in \mathbb{R}_+.
\end{align}
where, $\langle f, g \rangle_T=\int_0^Tf(t)^Tg(t)dt$.
\end{definition}

Remark: In the case of  equality in the inequalities  \eqref{eq:storage1} and \eqref{eq:inerpassive}, the system is said to be \textit{lossless}. 
%\item There exists a constant $\alpha$ such that:
%\begin{align}\label{eq:inerpassive}
%\langle {\Sigma}u, u\rangle_T \geq\alpha,
%\end{align}
%where, $\langle f, g \rangle_T=\int_0^Tf(t)^Tg(t)dt$.
%\end{itemize}

%Furthermore, a system $\Sigma$  is said to be input strictly passive if  there exist $\beta>0$ such that: 
%\begin{align}\label{eq:inerspassive}
%\langle {\Sigma}u, u\rangle_T \geq \alpha + \beta \langle u, u \rangle_T.
%\end{align}
The stability of the  feedback interconnection between passive systems is  a fundamental result in passivity theory (e.g., \cite{VanderSchaft2012}).  That is, the negative feedback interconnection between a passive system  $\Sigma_1$ and strictly passive  $\Sigma_2$,  as shown in  Fig. \ref{fig1},  is stable  feedback interconnection. Also, the closed loop system from $r$ to $y_1$ is passive. 

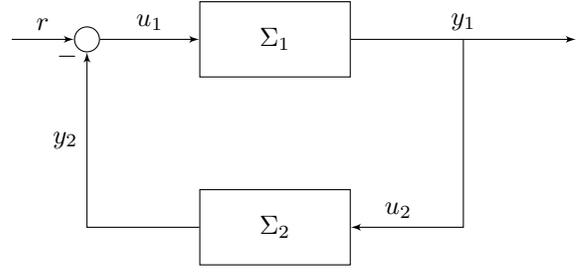
\begin{figure}
\begin{center}
\begin{tikzpicture}[auto, node distance=2.5cm,>=latex']
    % We start by placing the blocks
    \node [input, name=input] {};
    \node [sum, right of=input] (sum) {};
   \node [block, right of=sum] (system) {$\Sigma_1$};
    \node [output, right of=system] (output) {};
    \node [block, below of=system] (controller) {$\Sigma_2$};
    \draw [draw,->] (input) -- node {$r$} (sum);
    \draw [->] (sum) -- node {$u_1$} (system);
    \draw [->] (system) -- node [name=y] {$y_1$}(output);
    \draw [->] (y) |- node [above,pos=0.79]{$u_2$} (controller) ;
    \draw [->] (controller) -| node[pos=0.99] {$-$} 
        node [near end] {$y_2$} (sum);
\end{tikzpicture}
\end{center}
        \caption{Negative  feedback interconnection. }\label{fig1}
\end{figure}

The following definition  defines the \textit{$\delta$-passive} and  the \textit{$\delta$-anti-passive} dynamics. The definition was introduced in \cite{Fox2013}, where the connection between passivity property  and games were established.  
\begin{definition}
The input-output operator $S : U\longrightarrow Y $ is  said  to be 
\begin{itemize}
\item $\delta$-passive if  there exists a constant $\alpha$ such that:
\begin{align*}
\langle (\dot{Su}), \dot{u}\rangle_T \geq\alpha , \forall u \in U, T\in \mathbb{R}_+. 
\end{align*}
\item input strictly $\delta$-passive if  there exists a constant $\alpha$ and $\beta >0$ such that
\begin{align*}
\langle (\dot{Su}), \dot{u}\rangle_T \geq\alpha +\beta \langle (\dot{u}), \dot{u}\rangle_T, \forall  u \in U, T\in \mathbb{R}_+.
\end{align*}
\item $\delta$-anti-passive if $-S$ is $\delta$-passive.
\end{itemize}
\end{definition}

The following proposition derives the relationship  between  passivity and $\delta$-passivity for linear systems.  
Consider the linear time invariant (LTI)  system with $R$ as an input-output mapping and the  following  state space representation:
\begin{equation}
\begin{aligned}
\label{eq:PR1x}
& \dot{x}(t)=Ax(t)+Bu(t)\\
 & y(t)=Cx(t)+Du(t), \;\;\;\;\; x(0)=x_0,
\end{aligned}
\end{equation}
where, $A \in \mathbb{R}^{n \times n},B \in \mathbb{R}^{n \times
m},C \in \mathbb{R}^{m \times n},$ and $D \in \mathbb{R}^{m \times
m}$.

\begin{proposition}\label{lem:p-dp}
The input-output mapping $R$ is passive if and only if it is  $\delta$-passive.
\end{proposition}
\begin{proof}
The positive real lemma (e.g., \protect\cite[p. 70]{lozano2013}) implies that  the  input-output mapping $R$ with the  associated transfer function  $R(s)=\begin{bmatrix}
\begin{array}{c|c}
A & B \\ \hline C & D
\end{array}
\end{bmatrix}$  is positive real if and only if  exist a matrix $P>0$, such that the following linear matrix inequality (LMI):
\begin{align}\label{eq:lmi1}
\begin{bmatrix}
A^T P+PA&PB-C^T\\
B^TP-C&-(D+D^T)
\end{bmatrix}\leq0,
\end{align}
 holds. 
 
Since $\delta$-passive systems are defined for continuously deferentiable inputs and outputs, then the   state space representation for mapping from $\dot{u}$ to  $\dot{y}$ can be obtained from   the derivative  of the state space representation \eqref{eq:PR1x} as follows:
\begin{equation}
\begin{aligned}
 \dot{\widehat{x}}(t)&=A\widehat{x}(t)+B\widehat{u}(t) \label{eq:PR2x}\\
  \widehat{y}(t)&=C\widehat{x}(t)+D\widehat{u}(t), \;\;\;\;\; \widehat{x}(0)=\widehat{x}_0,
\end{aligned}
\end{equation}
where,  $\widehat{x}=\dot{x},\widehat{y}=\dot{y}$ and $\widehat{u}=\dot{u}$.  It is clear that the state space representation \eqref{eq:PR2x}  has the same input-output mapping   $R=\begin{bmatrix}
\begin{array}{c|c}
A & B \\ \hline C & D
\end{array}
\end{bmatrix}$, which implies that the  LMI \eqref{eq:lmi1} is the same  for the system \eqref{eq:PR2x}. This implies that the input-output mapping $R=\begin{bmatrix}
\begin{array}{c|c}
A & B \\ \hline C & D
\end{array}
\end{bmatrix} $ is passive from  $\dot{u}$ to  $\dot{y}$ (i.e., $\delta$-passive) if and only if exist a matrix $P>0$ the LMI \eqref{eq:lmi1} holds. This completes the proof.  
\end{proof}

\subsection{Stable Games}

A game $G$, in general, consist of three basic elements.  Number of \textit{players} $N$:   are  the  decision makers in  the game context. 
\textit{Strategies} $S$:  are  the set   of actions that a particular player will play  given a set of conditions and circumstances that will emerge in  the game being played. 
\textit{Payoff} $P$: is the reward which a player receives from playing  at a particular strategy.

A  population game  has a set of  strategies $S =\{1, 2, ..., m\}$ and  a set of strategy distributions 
$X=\{x: \;\; \Sigma_{i \in S}x_i=1\}$. Since strategies lie in the simplex, admissible changes in strategy are restricted to the tangent space $TX=\{z\in  \mathbb{R}^m : \sum_{i\in s} z_i=0\}$. 

A state $x^*\in X$ is a \textit{Nash equilibrium} if  each strategy in the support of $x$ receives the maximum payoff available to the population.

The definition of \textit{stability} in this context implies that there is an evolutionary stable state i.e., rest point, where the distance between  the population distribution  and this rest point decreases along the population trajectories, i.e., the population converge to this stable state. Therefore, \textit{unstable} feedback loop between learning rule and game means  that the feedback will not converge to a rest point. 

We now recall the definition of \textit{stable games} for continuously differentiable games as presented in \cite{Hofbauer2009,Fox2013}. First, define $F(x) : S\longrightarrow \mathbb{R}^m$ to be  the payoff function that  associate each strategy distribution in $S$ with a payoff vector so that $F_i: S\longrightarrow \mathbb{R}^m$ is the payoff to strategy $i\in S$. Also, define $DF(x)$ to be the Jacobian matrix of $F(x)$. 

\begin{definition}
Suppose that $F(x)$ is continuously differentiable. Then $F$ is  said to be  \textit{stable game} if 
\begin{align*}
z^TDF(x)z\leq 0, \;\;\;\;\;\forall z \in TX.
\end{align*}  
\end{definition}
More detailed discussions  on stable games and there dynamics  are given in \cite{Hofbauer2009,Fox2013}.

%\begin{itemize}
%\item Section II: 
%\begin{itemize}
%\item Passive system (state space only? I/O only?)
%\item Delta passivity
%\item Standard stable game (continuum only!)
%\item Higher-order stable game
%\end{itemize}
%\end{itemize}
% 
 The relationship between passivity and stable games is established in \cite{Fox2013}. That is, let $\mathbb{X}$ denote  locally Lipschitz $X$-valued functions over $\mathbb{R}_+$ and $\mathbb{P}$ denote  locally Lipschitz $\mathbb{R}^m$ -valued functions over $\mathbb{R}_+$.
\begin{theorem} [\cite{Fox2013}]
A continuously differentiable stable game mapping $\mathbb{X}$  to $\mathbb{P}$  is $\delta$-anti-passive game. 
\end{theorem}

\subsection{ Higher-order dynamics and games}
In the continuous-time  standard evolutionary dynamics and games, the game is static mapping from strategies $X$  to payoffs $P$,
\begin{equation*}
P=F(x),
\end{equation*}
and the dynamics are restricted first order mapping from payoffs $P$ to strategies $X$,
\begin{equation*}
\dot{x}=V(x,p).
\end{equation*}
The dynamical view of this feedback loop can be extend to a mapping of strategy trajectories to payoff trajectories. This viewpoint allows us to introduce  generalized  forms of dynamics and games, such as higher-order  dynamics and games, to generate these trajectories. 

Higher-order  dynamics can be introduced---independent of the game---through  auxiliary states to the the first order dynamics \cite{Shamma2005,Arslan2006}, which can be interpreted as path dependency. Also, similar  higher-dynamics, but depends on the game,  can be obtained by the direct derivative of the first order dynamics \cite{Laraki2013}.  It has been shown in \cite{Hart2003,Sato2002,Shamma2005,Arslan2006,Laraki2013} that the  modification of the standard dynamics   can exhibit qualitatively different behaviors. One form of  generalized higher-order dynamics obtained by an auxiliary state $z$  is given as follows:
\begin{equation*}
\begin{aligned}
\dot{z}&=f(z,p)\\
x&=g(z).
\end{aligned}
\end{equation*}

%{\color{red}
% Suppose that the first order \emph{dynamical system} $\dot x(t) = V(x(t))$, $t\geq0$  has a unique solution $x(t)$. 
%Hence, we will say that $x^*$ is:
%\begin{itemize}
%\item \emph{stationary point} if $V(x^*) = 0$.
%\item
%\emph{locally  stable} if, for every $R>0$ there exist $r>0$ such that:
%\begin{align*}
%\Vert x(0)-x^* \Vert<r \Rightarrow \Vert x(t)-x^* \Vert<R, \;\;\ t\geq0,
%\end{align*}
%\item
%\emph{locally asymptotically stable} if it is locally 
% stable and,
% \begin{align*}
%\Vert x(0)-x^* \Vert<r \Rightarrow \lim_{t\rightarrow \infty} x(t)=x^*,
%\end{align*}
%\item
%\emph{globally asymptotically stable} if it is locally asymptotically stable for all $x(0)$.
%\end{itemize}}

%The following definition defines higher-order dynamics  through direct derivative of the first order dynamics\cite{Laraki2013}.
%\begin{definition}
%Higher-order dynamics of the form $x^{(n)} = V$ is  defined using  the recursive form as follows:
%\begin{equation}
%\begin{aligned}
%\label{eq.ndyn}
%\dot x(t)	&= x^{1}(t)\\
%\dot x^{1}(t)	&= x^{2}(t)\\
%			&\vdots\\
%\dot x^{n-1}(t) &= V(x(t),x^{1}(t),\dotsc,x^{n-1}(t)),
%\end{aligned}
%\end{equation}
%where the $n$-th order dynamical system  correspond to $n$-tuples of the form $(x,x^{1},\dotsc,x^{n-1})$ with the  initial state $(x(0), \dot x(0),\dotsc, x^{(n-1)}(0))$.
%\end{definition}

 Similarly,  static games  can be generalized by introducing  internal dynamics into the game. This concept is illustrated in \cite{Fox2013} through dynamically modified payoff function coupled with the static game. Therefore, we view  the  higher-order games as a generalization of   standard games  by introducing internal dynamics into the game, i.e., \textit{dynamical system} mapping from strategies $X$  to payoffs $P$.

\section{Necessity Conditions for Stable Interconnection with Passive Systems}
\label{sec:subspace}

This section establishes the following necessity condition: If a system $\Sigma_1$ is stable in the negative feedback interconnection with all passive systems, then, $\Sigma_1$ must be passive. To prove this statement, we recall a necessity result for a small gain condition  \cite{Zhou1996}. We first consider linear systems followed with a linearization based result of nonlinear systems.

\subsection{Small Gain Theorem}\label{subsec:SGTH}

The following proposition provides the necessity  conditions for feedback interconnected systems with  small gain property. The result is part of the small gain theorem provided in \cite{Zhou1996}.

\begin{proposition}[\cite{Zhou1996}] \label{th:sgth}
For any stable system $S$,  associated transfer function matrix $S(s)=\begin{bmatrix}
\begin{array}{c|c}
A_s & B_s \\ \hline C_s & D_s
\end{array}
\end{bmatrix}$ with   $H$-infinity norm  $\Vert S(j\omega) \Vert_\infty   >1$,  
there exists a transfer function matrix $\Delta(s)=\begin{bmatrix}
\begin{array}{c|c}
A_\Delta & B_\Delta \\ \hline C_\Delta & D_\Delta
\end{array}
\end{bmatrix}$ with   $H$-infinity norm $\Vert \Delta (j\omega) \Vert_\infty   <1$,  such that the closed loop feedback is unstable. \QEDB
\end{proposition}

For completeness, we recall the proof in  \cite{Zhou1996} in order to utilize the explicit construction of a destabilizing $\Delta(s)$. 
The  closed loop feedback between  $[S(s), \Delta(s) ]$ is unstable if  $\det(I+S(s)\Delta(s))=0$, i.e., it is sufficient to construct  $\Delta(s)$ with  $\Vert \Delta (j\omega) \Vert_\infty   <1$ such that $\mathrm{det}(I+S(s)\Delta(s))=0.$

Suppose that $\omega_0 \in (0,\infty) $ where $\Vert S(j\omega_0) \Vert_\infty  =\sigma_1 >1$. Let  the singular value decomposition  (SVD) of $S(j\omega_0)$ to be $U \Sigma V^*$, where  $U=\begin{bmatrix}
u_1 & u_2 & \cdots
\end{bmatrix}$  and $V=\begin{bmatrix} 
v_1 & v_2 & \cdots
\end{bmatrix}$ are  unitary matrices. We can rewrite $u_1^{*}$ and $v_1$ as  
$$u_1^{*}=\begin{bmatrix}
u_{11} e^ {j\theta_1} & u_{12} e^ {j\theta_2}& \cdots
\end{bmatrix}$$
and 
$$v_1=\begin{bmatrix}
v_{11} e^ {j\phi_1} \\ v_{12} e^ {j\phi_2} \\ \vdots
\end{bmatrix},$$
where $\theta_i, \phi_i \in (-\pi,0)$ and 
$$\theta_i=\measuredangle \dfrac{\beta_i-j\omega_0}{\beta_i+j\omega_0}$$
$$\phi_i,=\measuredangle \dfrac{\alpha_i-j\omega_0}{\alpha_i+j\omega_0}.$$

Now define 
$$\Delta(s)=\dfrac{1}{\sigma_1}\begin{bmatrix}
v_{11} \dfrac{\alpha_1-s}{\alpha_1+s} \\ v_{12} \dfrac{\alpha_2-s}{\alpha_2+s} \\ \vdots
\end{bmatrix}\begin{bmatrix}
u_{11} \dfrac{\beta_1-s}{\beta_1+s}  & u_{12} \dfrac{\beta_2-s}{\beta_2+s} & \cdots
\end{bmatrix}.$$

This construction of $\Delta(s)$ ensures that 
$$\Vert \Delta (j\omega) \Vert_\infty  =\dfrac{1}{\sigma} <1.$$
It follows that at $s=j\omega_0$, $ \Delta (j\omega_0)   =\dfrac{1}{\sigma} v_1 u_1^*$ and hence 
$$\det(I+S(j\omega_0)\Delta j\omega_0)=\det(I+U \Sigma V^*\dfrac{1}{\sigma} v_1 u_1^*)=0.$$
This completes the proof. 

In other words, Proposition \ref{th:sgth} can be read as follows: if a system  $\Sigma_1$ is stable in the  feedback interconnection with all small gain systems, then $\Sigma_1$ must  have  small gain property.

Following  Proposition \ref{th:sgth},  we recall the relationship  between passivity and small gain property  (e.g., \cite{VanderSchaft2012}), in order   to provide similar result for  passive systems. 

The passivity-small gain  relationship  is known as follows:
\begin{lemma}[\cite{VanderSchaft2012}]
Suppose that $(G(s)+I)$ is invertible, then an LTI  system  $S(s)$ has  small gain property, i.e., $\Vert S \Vert_\infty  <1$, if and only if $G(s)$ is  passive system, where,
\begin{equation}\label{eq:smg-pass}
S(s)=(G(s)-I)(G(s)+I)^{-1} . 
\end{equation}  
\end{lemma}

%\subsection{Passivity}
Now, we are ready to provide the necessity part for linear  passive systems.  
\begin{theorem}\label{th:passivety}
For any  LTI stable strictly non-passive  system $G(s)$, 
there exists   strictly passive  system $R(s)$ such that  the closed loop feedback between $[G(s), R(s) ]$ is unstable. \QEDB
\end{theorem}

\begin{proof}
Let $G$, $S$, $\Delta$ and $R$ denote $G(s)$, $S(s)$, $\Delta(s)$ and $R(s)$ respectively. Suppose that $G$ is  non-passive transfer function with $G+I$ invertible. Using \eqref{eq:smg-pass}, it implies that there exist $S$ such that:
\begin{equation*}
S=(G-I)(G+I)^{-1} \text{ and }\Vert S \Vert _\infty > 1.
\end{equation*}
Using Theorem \ref{th:sgth}, it follows that these exist $\Delta$ such  $\Vert \Delta \Vert_\infty  <1$ and the  closed loop $[S,\Delta]$ is unstable. This implies that:
\begin{equation*}\label{eq:det0}
\det(I+S \Delta)=0.
\end{equation*}
However, $\Vert \Delta \Vert_\infty <1$ implies that there exist passive system $R$ such that:
\begin{equation}
\Delta=(R-I)(R+I)^{-1}.
\end{equation}
This implies that:
\begin{align*}
&\det(I+S \Delta)=0 \\ \notag
\Rightarrow &\det\left(  I+(G-I)(G+I)^{-1}(R-I)(R+I)^{-1}\right) =0\\  \notag
\Rightarrow &\det\left(  I+(G-I)(G+I)^{-1}(R+I)^{-1}(R-I)\right) =0\\  \notag
\Rightarrow &\det((R+I)(G+I))\\&\;\;\; \times \det\left(  (R+I)(G+I)+(R-I)(G-I)\right) =0\\  \notag
\Rightarrow &\det\left(  RG+G+R+I+RG-G-R+I\right) =0\\  \notag
\Rightarrow &\det\left(  2RG+2I\right) =0\\  \notag
\Rightarrow &\det\left(  RG+I\right) =0, \notag
\end{align*}
which implies  the closed loop $[G,R]$ is unstable. This completes the proof. 

%Note:- 
%\begin{align*}
%det\left( X+BA\right) =det(X)det\left(I+AX^{-1}B\right) 
%\end{align*}
\end{proof}

Remark:  The statement of Theorem \ref{th:passivety}  is equivalent to the following statement:
\textit{Suppose that   $G$ is  LTI stable system and forms stable feedback interconnection with all passive systems, then  $G$  must be passive.}

\section{Passivity Analysis for Higher-order dynamics  and  games}\label{sec:HOG}
In this section, we focus on   passivity  analysis of the  higher-order evolutionary dynamics and games. As shown in Fig. \ref{fig:game-dynamic}, games and evolutionary dynamics can be illustrated as a feedback interconnection.  We will show   that if an evolutionary dynamic or a learning rule  is non-$\delta$-passive, then it is possible to  construct higher-order   $\delta$-anti-passive game that results in instability. In other words, a learning rule   results in stability for all higher-order  $\delta$-anti-passive games if and only if the learning rule is  $\delta$-passive.

\begin{figure}
\begin{center}
\begin{tikzpicture}[auto, node distance=3cm,>=latex']
    % We start by placing the blocks
    \node [input, name=input] {};
    \node [sum, right of=input] (sum) {};
   \node [block, right of=sum] (system) {Higher-order  Learning Rule};
    \node [output, right of=system] (output) {};
    \node [block, below of=system] (controller) {Higher-order  Game};
    %\draw [draw,->] (input) -- node {$r$} (sum);
    \draw [->] (sum) -- node {$P$} (system);
    \draw [->] (system) -- node [name=y] {$X$}(output);
    \draw [->] (y) |- node [above,pos=0.79]{$X$} (controller) ;
    \draw [->] (controller) -| node[pos=0.99] {$+$} 
        node [near end] {$P$} (sum);
\end{tikzpicture}
\end{center}
        \caption{Feedback interconection of higher-order dynamic and generlaized game, where $P$ is the payoffs and $X$ is the strategies. }\label{fig:game-dynamic}
\end{figure}
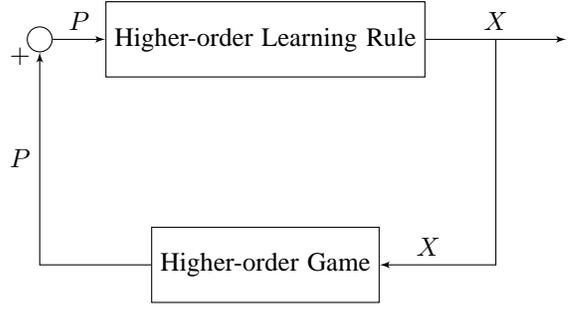

The following theorem provides necessity conditions for  passive feedback interconnections of Higher-order  learning rule and  higher-order  game  as an analogy to Theorem \ref{th:passivety}.  

\begin{theorem}\label{th:main}
If the linearization of the Higher-order  learning rule is not $\delta$-passive, then there exist  $\delta$-anti-passive game that results in unstable  positive feedback interconnection. 
\end{theorem}
\begin{proof}

Passive systems in negative feedback interconnection are equivalent to  anti-passive systems  in positive feedback interconnection. Also, Proposition \ref{lem:p-dp} implies that passivity and $\delta$-passivity are equivalent for linear systems. That is, if  the learning rule is not $\delta$-passive, then it is not passive.  Therefore, using  Theorem \ref{th:passivety}, it follows  that it is  possible to  construct   $\delta$-anti-passive to destabilize  a given  non-$\delta$-passive learning rule. 
\end{proof}

In the next subsections, we will focus on the passivity analysis of  two familiar classes of higher-order dynamics, Logit dynamics and replicator dynamics. 
\subsection{Second order Logit dynamics }

 Roughly speaking, the logit learning rule can be considered  as a noisy version of the  best response dynamics. In this dynamic, the change in the strategies of the players depend on the level of their  knowledge about the game and  the strategies currently played \cite{McFADDEN1974}.

The second order logit dynamics can be obtained  by introducing an auxiliary state in the payoff function as follows:
\begin{subequations}\label{eq:logitD}
\begin{align} 
\dot{x_i}&=-x_i+\frac{e^{\widehat{p}_i}}{\sum_j e^{\widehat{p}_j}}\label{eq:logitD1}\\
\dot{\widehat{p}_i}&=-\widehat{p}_i+p_i,\label{eq:logitD2}
\end{align}
\end{subequations}
where, $x$ is the state vector, $p$ is the payoff  and  $\widehat{p}$ is the introduced auxiliary state.
The equilibrium conditions are $ \widehat{p}=p$ and $x_i=\sigma_{\text{max}}(\widehat{p}_i)$.

First, we show that the  linearization of \eqref{eq:logitD}  is non-passive system. Hence, according to our results, it is possible to construct   higher-order passive game that results in instability with the second order Logit dynamics.  The linearized system is given as follows:
\begin{subequations}\label{eq:logitDL}
\begin{align}
\delta \dot{x}&= A_x \delta x+\beta\delta \widehat{p} \label{eq:logitDl1}\\
\delta \dot{\widehat{p}} &= A_p\delta \widehat{p}+ B_p\delta p,\label{eq:logitDl2}
\end{align}
\end{subequations}
where, $\delta x =x-x^*$ is the deviation from  equilibrium, $A_x=A_p=-I$,  $\beta =\frac{1}{m}\begin{bmatrix} 1&-1&-1&\cdots\\-1&1&-1&\cdots\\ \vdots&\vdots&\vdots&\vdots\\-1&-1&\cdots&1
\end{bmatrix}$, $m$ is the number of pure strategies and $B_p=I$. 
%The augmented dynamic matrix is $A=\begin{bmatrix} -I&\beta\\0&-I
%\end{bmatrix}$ and $B=\begin{bmatrix} 0\\I
%\end{bmatrix}$. 
The linearized logit dynamics   \eqref{eq:logitDL} is a linear second order dynamical system and its  transfer function has relative degree equal two, i.e., the  linearized logit dynamics    is not passive.

Before we analyze the linearized   logit dynamics   \eqref{eq:logitDL} and  construct a higher-order  game that results in unstable feedback interconnection, we will reduce the linearized dynamics using the transformation  $\delta x=\delta x^*+N \delta w$,  $\delta p=\delta p^*+N\delta q$ and $\widehat{p}=N\xi$, where $\delta x^*$, $p^*$ is the equilibrium point,  $N$ is the null space of a vector of ones and satisfies $N^TN=I$. This projection ensure that the dynamics stays in the simplex. 
 The reduced linear dynamical system is given as follows:
 \begin{subequations}\label{eq:logitDLR}
\begin{align}
\begin{bmatrix}
\delta \dot{w}\\ \dot{\xi}
\end{bmatrix}&=\begin{bmatrix}
-I&N^T \beta N\\0&-I
\end{bmatrix}\begin{bmatrix}
\delta {w}\\ \xi
\end{bmatrix}+\begin{bmatrix}
0\\ I
\end{bmatrix}\delta q.
\end{align}
\end{subequations}

Now, consider the case where $i=1,2,3$ and let $\delta W=\begin{bmatrix}
\delta {w}\\ \xi
\end{bmatrix}$. The reduced dynamical system \eqref{eq:logitDLR} is given as follows:
%\begin{subequations}\label{eq:logitDLR2}
%\begin{align}
%\delta\dot{ w}=& N_1^T\begin{bmatrix}
%-1&0&\frac{1}{4}&-\frac{1}{4}\\0 & -1&-\frac{1}{4}&\frac{1}{4}\\0&0&-1&0\\0&0&0&-1
%\end{bmatrix}N_1 \delta w\\&+N_1^T \begin{bmatrix}
%0& 0  \\
%0& 0\\
%1&0\\
%0&1\\
%\end{bmatrix}
%N_2 \delta q
%%\delta y=&\begin{bmatrix}
%% 1&0&0&0\\
%%0&1&0&0\\
%%\end{bmatrix}N_1^T\delta w.
%\end{align}
%\end{subequations}
\begin{subequations}\label{eq:logitDLR3}
\begin{align}
\delta\dot{ W}&= \begin{bmatrix}
-1&0&\frac{2}{3}&0\\ 0&-1&0&\frac{2}{3}\\0&0&-1&0\\0&0&0&-1
\end{bmatrix}\delta W+\begin{bmatrix}
0 &0\\0&0 \\
1&0\\0&1
\end{bmatrix}
\delta q\\
\delta y&=\begin{bmatrix}
 1&0&0&0\\0&1&0&0
\end{bmatrix}\delta W,
\end{align}
\end{subequations}
which is  non-passive dynamical system $G(s)$ with the following transfer function matrix:
\begin{align}G(s)=\begin{bmatrix}
\frac{2}{3(s+1)^2}&0\\0&\frac{2}{3(s+1)^2}
\end{bmatrix}.
\end{align}
The construction mechanism provided in this paper will be employed  to construct a passive system (game) that result in unstable feedback loop  with  the second order logit dynamics.  The internal dynamics   of the constructed passive  game is given   as follows:
\begin{subequations}\label{eq:gg}
\begin{align}
& \dot{ z}=A_gz+B_g u_g\label{eq:xdotgg}\\
 &  y_g=C_g z+    D_gu_g,\label{eq:ygg}
\end{align}
\end{subequations}
where, 
\begin{align*}
A_g&=\begin{bmatrix}
0  & -0.8608     &    0   &      0\\
    1.0000  & -1.0791     &    0     &    0\\
         0   &      0     &    0  & -0.8608\\
         0     &    0   & 1.0000   &-1.0791
\end{bmatrix},
\\ B_g&= \begin{bmatrix}
0&0\\10&10\\0&0\\10&10
\end{bmatrix},\;\;C_g=\begin{bmatrix}0&16&0&0\\0&0&0&16
\end{bmatrix},\\
D_g&=\begin{bmatrix}
5.2020  & -4.7980\\
   -4.7980  &  5.2020
\end{bmatrix}.
\end{align*}
%\begin{figure}[H]
%\begin{center}
%\centering{\includegraphics[height=5.1 cm]{passLD1.eps}}
%\caption{Bode plot of the constructed  passive system in order to destiblize  the non-passive second order logit dynamics.  }\label{fig:passLD1}
%\end{center}
%\end{figure}
% As shown in Fig. \ref{fig:passLD1}, the phase  lies between $(\frac{\pi}{2},-\frac{\pi}{2})$, which implies that 
 The  internal dynamics of the game  \eqref{eq:gg}  is  passive mapping from $u_g$ to $y_g$. The feedback   interconnection  between the game and   the logit dynamic implies that $u_g=\delta w$ and $y_g=\delta q$. 
Using Proposition \ref{lem:p-dp}, it follows that the above game  is $\delta$-passive from $\delta \dot{w}$ to $\delta \dot{q}$. This implies that there is a storage function $L(z,q)$ such that:
\begin{align}\label{Logit:storage}
L(z,q)-L(0,0)\leq \int_0^t\delta \dot{w}^T(\tau) \delta \dot{q}(\tau)d\tau. 
\end{align}
  Also, the transformations $X=X^*+N\delta w$ and $P=P^*+N\delta q$, implies that  $\dot{X}^T\dot{P}=\delta\dot{w}N^TN \delta\dot{q}=\delta \dot{w}\delta\dot{ q}$. Using \eqref{Logit:storage},  it follows that the game \eqref{eq:gg} is  $\delta$-passive from  $\dot{X}$ to $ \dot{P}$.  Hence,   the higher-order  $\delta$-passive   game  from strategies $X$ to payoff $ P$ is  given as follows:
\begin{subequations}\label{eq:gg2}
\begin{align}
X&=X^* + N\delta w\\
 \dot{ z}&=A_gz+B_g\delta w\label{eq:xdotgg2}\\
   P&=P^* + N(C_g z+    D \delta w).\label{eq:ygg2}
\end{align}
\end{subequations}
Proposition  \ref{lem:p-dp} implies  that the dynamics \eqref{eq:logitDLR3} is not $\delta$-passive. Also, using the fact that $\dot{X}^T\dot{P}=\delta \dot{w}\delta\dot{ q}$, it follows that the dynamics  \eqref{eq:logitDLR3} is not $\delta$-passive from  $\dot{X}$ to  $\dot{P}$. Hence, according to Theorem \ref{th:main}, the feedback interconnection between the constructed $\delta$-anti-passive and the non-$\delta$-passive dynamics result in instability. 
Figs.  \ref{logit-gs} and \ref{logit-g}  shows  the evolution of the states of the   feedback interconnection between the  constructed  game, which is   $\delta$-anti-passive, and  the second order non-$\delta$-passive logit dynamics \eqref{eq:logitD}. 
\begin{figure}[H]
%\begin{center}
\centering{\includegraphics[height=9. cm]{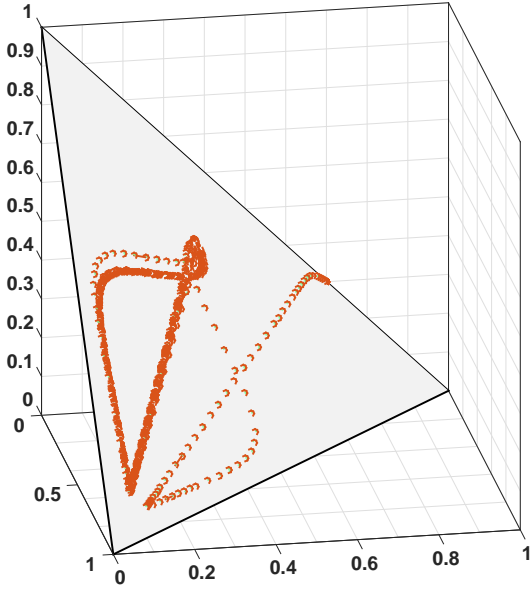}}
\caption{The evolution of the states  of the  constructed  $\delta$-anti-passive game in feedback loop with the non-$\delta$-passive second order logit dynamics \eqref{eq:logitD} projected into the simplex.  }\label{logit-gs}.
%\end{center}
\end{figure}

\begin{figure}[H]
\centering{\includegraphics[height=5.1 cm]{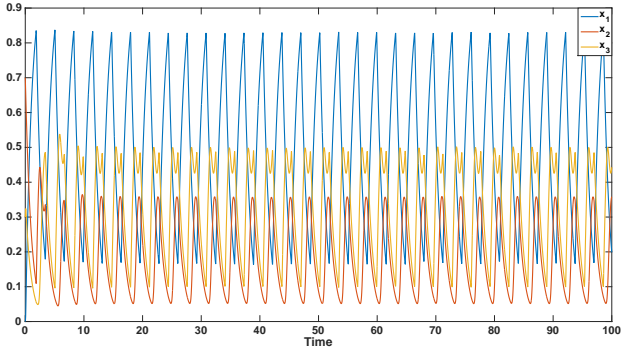}}
\caption{The evolution of the states of the  constructed $\delta$-anti-passive game in feedback loop with the non-$\delta$-passive second order logit dynamics \eqref{eq:logitD}.  }\label{logit-g}.
%\end{center}
\end{figure}

\subsection{Replicator dynamics}\label{sec:RD}
Replicator dynamics are important class of evolutionary dynamics, which emerges originally from system biology and nature evolution \cite{Schuster1983}. It provides way to represent selection among a population of diverse types.

%\subsubsection{First order replicator dynamics}

 First, we show that first order replicator dynamics are indeed passive dynamics. In particular, replicator dynamics belongs to special class of passive systems known as lossless systems.

%Consider a large population, $N$, of players with available   strategies
%set $S = {s_1, s_2, . . . , s_k }$. Let $n_i$ be the number of individual players using $s_i$. The population state is given by the vector  $x = (x_1, x_2, . . . , x_k )$, where $x_i=\frac{n_i}{N}$. Then
 The replicator dynamics is  given as follows:
\begin{align}\label{eq:RD}
\dot{x_i}=x_i(p_i-\sum_j x_j p_j),
\end{align}
where $p_i$ is the payoff for using strategy $i$.

Let $x^*$ to be  a Nash equilibrium for the dynamic. Define  $e_{x_i}=x_i-x^*_i$ to be the deviation from the equilibrium.   The following theorem shows that first order replicator dynamics from the payoff $p_i$ to the error $e_{x_i}$ belongs to a special class of passive systems named  lossless systems. 
\begin{theorem}
 Replicator dynamic mapping from $p_i$ to $e_{x_i}$ is passive lossless.
\end{theorem}
\begin{proof}
Using  $e_{x_i}=x_i-x^*_i$, the replicator dynamic equation \eqref{eq:RD} can be written as follows:
\begin{align}
\dot{e}_{x_i}=(e_{x_i}+x^*_i)(p_i-\sum_j (e_{x_j}+x^*_j) p_j).
\end{align}
Define the following storage function, 
\begin{align}
V(e_{x})=-\sum_i x^*_i\ln\frac{e_{x_i}+x^*_i}{x^*_i}.
\end{align}
Note  that $V(0)=0$, and 
\begin{align*}
V(e_{x})&=-\sum_i x^*_i\ln\frac{e_{x_i}+x^*_i}{x^*_i}\\&\geq -\ln\sum_i x^*_i\frac{e_{x_i}+x^*_i}{x^*_i}\\&=\sum_i x_i=1,
\end{align*}
i.e., $V(e_{x})\geq 0$.

Now, the derivative of the storage function is given as follows:
\begin{align*}
\dot{V}(e_{x})&=-\sum_i x^*_i\frac{1}{e_{x_i}+x^*_i}\dot{e}_{x_i}\\
&=-\sum_i x^*_i(p_i-\sum_j (e_{x_j}+x^*_j) p_j)\\
&=-\sum_i x^*_ip_i+\sum_i x^*_i(\sum_j e_{x_j}p_j+\sum_jx^*_j p_j)\\
&=-\sum_i x^*_ip_i+\sum_j e_{x_j}p_j+\sum_jx^*_j p_j\\
&=\sum_j e_{x_j}p_j.
\end{align*}
This  implies that replicator dynamics is passive (lossless) system.
\end{proof}

It is known that replicator dynamics can exhibit  different behaviors, that is stable, null stable  and unstable depending on the game, (e.g., \cite{Hofbauer2009}). For example one can show that  rock paper scissors game with replicator dynamics can generate these three different behaviors. These  behaviors can be seen as a consequences  of the lossless property of the replicator dynamics. 
%
%The payoff function of the rock paper scissors game is that $P=f(x)$ and $f(x)=Ax$, where 
%\begin{equation*}
%  A = \begin{pmatrix}
%  0 & -l &  \omega \\
%  \omega  & 0 &  -l \\
%  -l &  \omega &  0
%  \end{pmatrix}.
%\end{equation*}
%Here, $l,\omega$ are positive numbers. In the case of $\omega>l$ the rock paper scissors game  is strictly passive. However, in the case of $\omega<l$ the rock paper scissors game is non-passive. Moreover, in the standard case, i.e.,  $\omega=l$ the game is lossless. 

%In the next two subsections we will analyze two particular non-passive dynamics, \textit{second  order logit dynamics}  and \textit{second order replicator dynamics}.  The construction provided in this paper will be employed to generate generalized passive game that result in instability with these  non-passive second order dynamics.

%\subsubsection{Second order replicator  dynamics }

One  can  show that the linearization of the first order replicator dynamic results in a single  integrator. Now, we will show that  the second order replicators are non-passive dynamics, as the linearization result in double integrator (double poles at the origin).   Hence, according to our result in this paper, it is possible to construct higher-order game that result in instability with the second order replicator dynamics. 

The second order replicator dynamics can be obtained by introducing  auxiliary state $\widehat{p}$ in the payoff function. This results in the  following dynamics:
\begin{subequations}\label{eq:RD2order}
\begin{align}
\dot{x_i}&=x_i(\widehat{p}_i-\sum_j x_j \widehat{p}_j)\label{eq:RD2ordera}\\
\dot{\widehat{p}_i}&=p_i.\label{eq:RD2orderb}
\end{align}
\end{subequations}
The  equilibrium conditions are $ p_i^*=0$, $x_i^*=\sigma_{\text{max}}(\widehat{p}_i^*)$, and $\widehat{p}_i^*=1$ .

The linearization of the replicator dynamics \eqref{eq:RD2order} is given as follows:
%\begin{subequations}\label{eq:RDL}
\begin{align*}
\delta \dot{x}&= A \delta x+\beta\delta \widehat{p} \\%\label{eq:RDla}\\
\delta \dot{\widehat{p} }&= \delta p,%\label{eq:RDlb}
\end{align*}
%\end{subequations}
where, $A=\begin{bmatrix}
-x_1^* &-x_1^*&\cdots& -x_1^*\\
-x_2^* &-x_2^*&\cdots& -x_2^*\\
\vdots&\vdots&\vdots&\vdots\\
-x_n^* &-x_n^*& \cdots& -x_n^*
\end{bmatrix}$, $\beta=I-x^*x^{*T}$ and $x^*$ is any point in the simplex. The reduced system can be obtained using  the transformation $\delta x=\delta x^*+N \delta w$,  $\delta p=\delta p^*+N\delta q$ and $\delta \widehat{p}=N\xi$ as follows:
\begin{align*}
\delta \dot{w}&= N^TA N\delta w+N^T\beta N \xi\\%\label{eq:RDla}\\
 \dot{\xi }&= \delta q.%\label{eq:RDlb}
\end{align*}
Now, consider the case where $n=3$ and $x_i^*=\frac{1}{3}$,
\begin{subequations}\label{eq:RDLR3}
\begin{align} 
\begin{bmatrix}\delta \dot{w}\\
 \dot{\xi }
\end{bmatrix} &= \begin{bmatrix} 0&0&1&0\\
0&0&0&1\\
0&0&0&0\\
0&0&0&0
\end{bmatrix} \begin{bmatrix}\delta w\\
\xi
\end{bmatrix} +
\begin{bmatrix}0 &0\\
0&0\\1&0\\0&1 
\end{bmatrix} \delta q \\
\delta y&= \begin{bmatrix}1&0&0&0\\0&1&0&0
\end{bmatrix} \begin{bmatrix}\delta w\\
\xi
\end{bmatrix}.
\end{align}
\end{subequations}
The system \eqref{eq:RDLR3}  is a double integrator, which is not  passive. Accordingly, one can construct higher-order passive game that results in instability with the second order replicator dynamics. For instance,  the higher-order  game  can be constructed  as follows:
%\begin{subequations}\label{eq:RDg2}
\begin{align*}
X&=X^* + N\delta w\\
 \dot{ z}&=\begin{bmatrix}
-1&  0\\
    0  &  -1
\end{bmatrix}z+\begin{bmatrix}
1&0\\0&1
\end{bmatrix}\delta w\\%\label{eq:xdotRDg2}\\
   P&=P^* + N\begin{bmatrix}
 -1 &  0\\0&-1
\end{bmatrix}  z.\\%\label{eq:yRDg2}
\end{align*}
%\end{subequations}
Similarly, as in logit dynamics in the previous section, one can show that  the above game $\delta$-anti-passive from strategies $X$ to payoff $ P$.
 
Figs. \ref{fig:RD-g}  and \ref{RD-simplex-j} shows  the evolution of the states of the  positive feedback interconnection between the  constructed  game, which is   $\delta$-anti-passive, and  second  order non-$\delta$-passive replicator dynamic \eqref{eq:RD2order}.

\begin{figure}[H]
%\begin{center}
\centering{\includegraphics[height=5 cm]{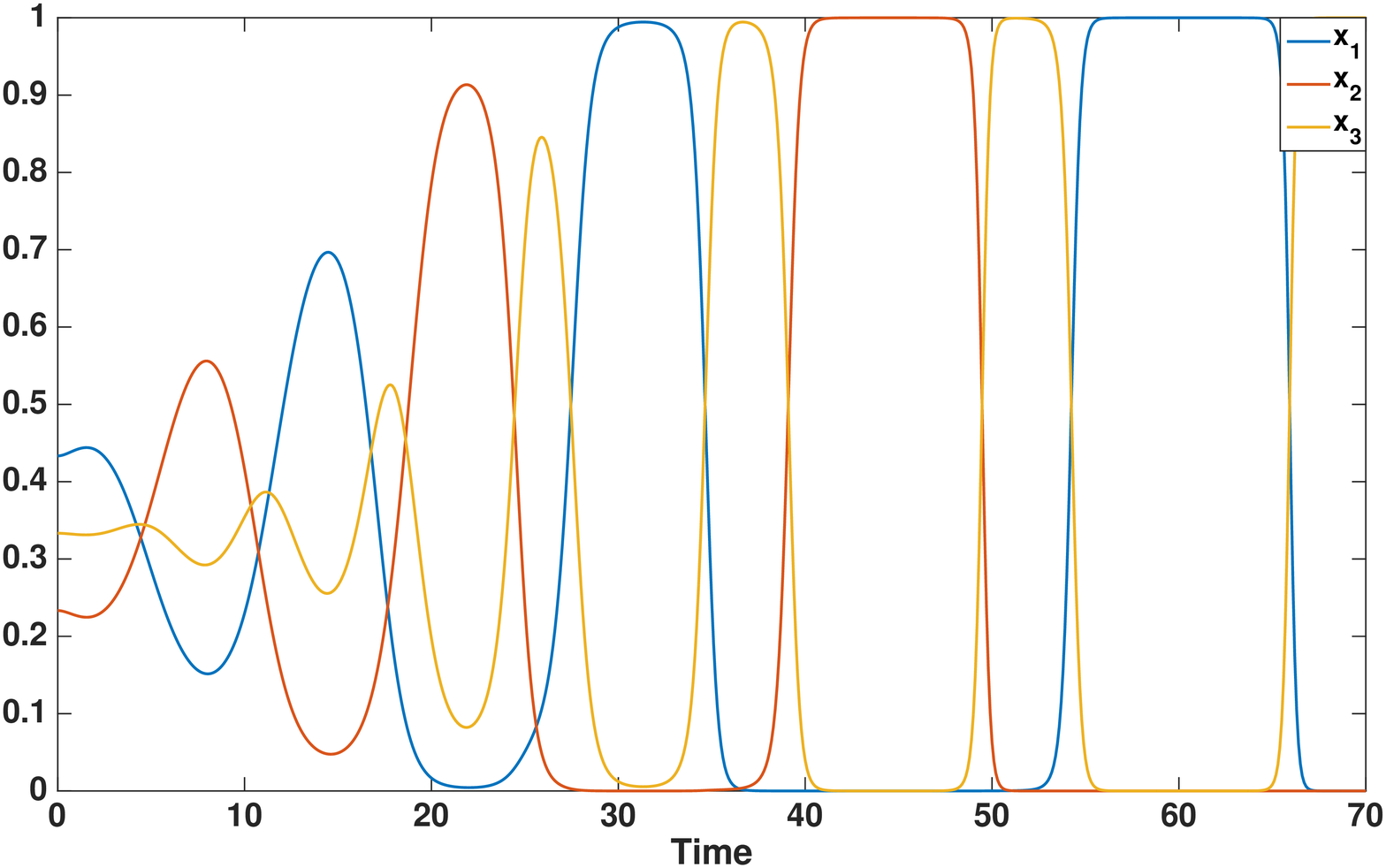}}
\caption{The evolution of the states of the  constructed $\delta$-anti-passive game in feedback loop with the non-$\delta$-passive second order replicator dynamics \eqref{eq:RD2order}.  }\label{fig:RD-g}.
%\end{center}
\end{figure}

\begin{figure}[H]
%\begin{center}
\centering{\includegraphics[height=9 cm]{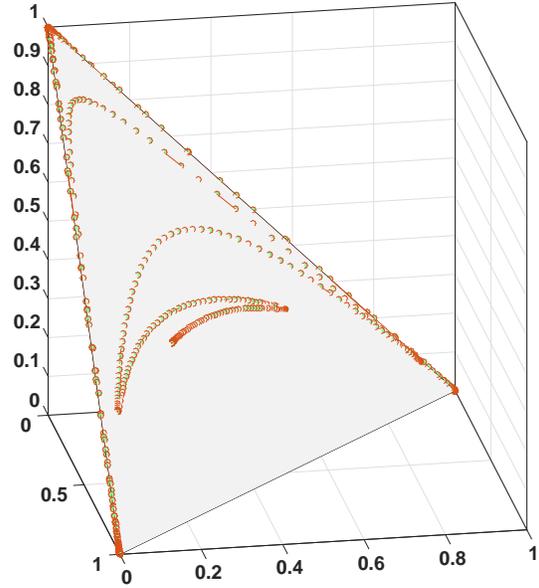}}
\caption{The evolution of the states  of the  constructed $\delta$-anti-passive game in feedback loop with the non-$\delta$-passive second order replicator dynamics \eqref{eq:RD2order} projected into the simplex.  }\label{RD-simplex-j}.
%\end{center}
\end{figure}

%where $p_i$ is the integral of  of the  payoff  Let $N=3$

\section{Concluding Remarks}

In this paper, passivity  analysis  for higher order dynamics and games has been presented. The  necessary conditions for evolutionary dynamics to exhibit stable behaviors for all higher-order passive games is provided. Methods from robust control analysis are used to show that if an evolutionary dynamic does not satisfy the passivity property, then it is possible to construct a higher-order passive game that results in unstable feedback loop. The results is employed to construct a higher-order passive games for two different  dynamics to illustrate the feedback passivity concept in games.

One can conclude similar result (under some detailed  conditions) for  the nonlinear passive  dynamics  by constructing linear non-passive system that result in instability with the linearization of the nonlinear system. In other words, the following conjecture is true under some detailed conditions: 
\textbf{Conjecture}: If a nonlinear system is locally non-passive, i.e., the linearization around an equilibrium point  $(x_0,u_0)$ is non-passive, then it is possible to construct a passive linear system that results in instability with the nonlinear system. This conjecture was illustrated in our discussion on   higher-order dynamics and games.

Similar investigation for first order dynamics was conducted in  \cite{Park2015}. They considered similar question raised  in this paper, but for class of  standard learning dynamics from passivity perspective.  Implications of stability for various passive dynamics both analytically and by means of numerical simulations was discussed.

\bibliographystyle{IEEEtran}
\bibliography{biblio}

\end{document}